\documentclass[11pt]{amsart}

\newcommand{\Zz}{\mathbb{Z}}
\newcommand{\Cc}{\mathbb{C}}
\newcommand{\Pp}{\mathbb{P}}

\newcommand{\Qq}{\mathbb{Q}}

\newcommand{\Hom}{\operatorname{Hom}}

\newcommand{\la}{\langle}
\newcommand{\ra}{\rangle}

\newcommand{\ii}{{\mathrm i}}

\newcommand{\Oo}{\mathcal{O}}

\def\dim{\mathop{\mathrm{dim}}\nolimits}

\def\Hom{\mathop{\mathrm{Hom}}\nolimits}

\def\mod{\mathop{\mathrm{mod}}\nolimits}

\makeatletter
\newtheorem*{rep@theorem}{\rep@title}
\newcommand{\newreptheorem}[2]{%
\newenvironment{rep#1}[1]{%
 \def\rep@title{#2 \ref{##1}}%
 \begin{rep@theorem}}%
 {\end{rep@theorem}}}
\makeatother

\newreptheorem{Thm}{Theorem}

\newreptheorem{Cor}{Corollary}

\newreptheorem{Con}{Conjecture}

\newtheorem{thm-int}{Theorem}

\theoremstyle{definition}
\newtheorem{Def-s}[Thm]{Definition}

\newtheorem{theorem}{Theorem}[section]

\newtheorem{remark}[theorem]{Remark}

\numberwithin{equation}{section}

\begin{document}

\title{Research announcement: equations of a fake projective plane}

\begin{abstract}
In this short note we announce explicit equations of a fake projective plane in its bicanonical embedding in $\Cc\Pp^9$.
\end{abstract}


\author{Lev Borisov}
\address{Department of Mathematics\\
Rutgers University\\
Piscataway, NJ 08854} \email{borisov@math.rutgers.edu}
\author{JongHae Keum}
\address{School of Mathematics\\
KIAS\\
}
\email{jhkeum@kias.re.kr}

\maketitle

\section{Introduction}
Fake projective planes are defined as complex projective surfaces of general type whose Betti numbers equal that of 
the projective plane. These surfaces have been classified in \cite{CS} as 50 complex conjugate pairs of free quotients of the
two-dimensional complex ball by explicit arithmetic subgroups. Unfortunately, such description does not lead to explicit equations,
since it is not clear how one can explicitly construct the relevant modular forms.

\medskip
In this research announcement we describe explicit equations of one pair of such surfaces in its bicanonical embedding in $\Cc\Pp^9$.
The surfaces are defined over the imaginary quadratic field $\Qq[\sqrt {-7}]$. The explanation of the calculation 
will appear soon in a subsequent paper.
 
\medskip
Let $\Cc \Pp^9$ be a projective space with homogeneous coordinates denoted by $(U_0,U_1,\ldots,U_9)$.
Consider the non-abelian group $G$ of order $21$ which is a semi-direct product $\Zz/7\Zz$ and $\Zz/3\Zz$.
We define its action on $\Cc\Pp^9$ by its action on the homogeneous coordinates by
\begin{equation}\label{group_action}
\begin{array}{l}
g_7(U_0:U_1:U_2:U_3:U_4:U_5:U_6:U_7:U_8:U_9) :=
\\
 \hskip 50pt 
(U_0: \xi^6 U_1:\xi^5 U_2:\xi^3 U_3:\xi U_4:\xi^2 U_5:\xi^4 U_6: \xi U_7:\xi^2 U_8:\xi^4U_9)\\
g_3(U_0:U_1:U_2:U_3:U_4:U_5:U_6:U_7:U_8:U_9) :=
\\ \hskip 50pt 
(U_0:U_2:U_3:U_1:U_5:U_6:U_4:U_8:U_9:U_7) \\
\end{array}
\end{equation}
where $\xi=\exp({\frac {2\pi \ii}7})$ is the primitive seventh root of $1$.

\begin{theorem}\label{main}
Eighty four cubic equations of Tables \ref{eqs1-24} and \ref{eqs25-84}
 give equations of a fake projective plane $Z$ in $\Cc\Pp^9$ embedded by its bicanonical linear system.
\end{theorem}

\begin{proof}
We used Magma to calculate the Hilbert series of $Z$ to give
$$
\dim H^0(Z,\Oo(k)) = \frac 12 (6k-1)(6k-2)
$$
for all $k\geq 0$. 

\medskip
We also verified that $Z$ is smooth. It is a somewhat delicate calculation because there are far too many minors of the Jacobian matrix and 
one needs to pick them carefully to show smoothness. Specifically, we used minors of the matrix of partial derivatives of the collections of seven
equations that are linearly independent to first order at the fixed points
$$
(U_0,\ldots,U_7,U_8,U_9)\in \{(0,\ldots,0,0,1),(0,\ldots,0,1,0),(0,\ldots,1,0,0)\}.
$$
This calculation was done in Magma with exact coefficients.

\medskip
Thus we have a smooth surface $Z$ and a very ample divisor class $D$ on it. We see that
$D^2=36$, $D K_Z=18$, and $\chi(Z,\Oo_Z)=1$. Note that this shows that $Z$ is not isomorphic to $\Cc\Pp^2$.

\medskip
We then used reduction modulo $263$ with $\ii\sqrt 7 = 16\mod 263$. We calculated (by Macaulay2) the projective resolution
of $\Oo_Z$ as 
$$
\begin{array}{c}
0\to
\Oo(-9)^{\oplus 28}\to
\Oo(-8)^{\oplus 189}\to
\Oo(-7)^{\oplus 540}\to\Oo(-6)^{\oplus 840}\to
\\
\to\Oo(-5)^{\oplus 756}\to\Oo(-4)^{\oplus 378}\to\Oo(-3)^{\oplus 84}\to\Oo\to\Oo_Z \to 0.
\end{array}
$$
By semicontinuity, the resolution is of the same shape over $\Cc$.
Since all of the sheaves $\Oo(-k),  ~k=3,\ldots,9$ are acyclic, we see that $h^1(Z,\Oo_Z)=h^2(Z,\Oo_Z)=0$ so we know all of the Hodge numbers of $Z$ other than $h^{1,1}(Z)$.

\medskip
To figure out this last Hodge number we used Macaulay to calculate $\chi(Z,\Oo(2K_Z))=10$ (again working modulo $263$).
This implies $K_Z^2=9$ and Noether's formula finishes the proof that $Z$ is a fake projective plane.

\medskip
We see that $2K$ is numerically equivalent to $D$.
We calculated 
$$\Hom(\Oo(K),\Oo(D))=0$$
by working modulo $263$ and semi-continuity.
This implies 
$$h^0(Z,\Oo(D-K))=0=h^2(Z,\Oo(2K-D)).$$
This implies that
$h^0(Z,\Oo(2K-D))\geq 1$, so $\Oo(2K)\simeq \Oo(D)$. So the fake projective plane $Z$ is embedded via a bicanonical embedding.
\end{proof}

\medskip
\begin{table}[htp]
\caption{Equations of the fake projective plane 1-24.}
\begin{center}
$$
\tiny{
\begin{array}{|cl|}
\hline
eq_1=&
(1 + \ii \sqrt{7}) U_1 U_2 U_3 
+
  (3 + \ii \sqrt{7}) U_4 U_5 U_6
+ 
8 (  U_1^2 U_5 + U_2^2 U_6+U_3^2 U_4 )
     \\[.8em]
     eq_2=&
     \frac{ (-1 + 3  \ii \sqrt{7})}4 
     U_0^3 
     + 
 \frac{(-7 + \ii \sqrt{7}) }2
 (2 U_1 U_2 U_3 + 8 U_4 U_5 U_6 - U_7 U_8 U_9)
  \\[.4em]&
     + (3 - \ii \sqrt{7}) U_0 ( 
   U_1 U_4 + U_2 U_5 + U_3 U_6) + 4 U_0 (U_1 U_7 + U_2 U_8 + U_3 U_9) 
        
         \\[.8em]
        eq_3=&
\frac {(11 - \ii \sqrt{7}) }{64} U_0^3 + 
\frac {(-1 - \ii \sqrt{7})}4 U_1 U_2 U_3 + 
  2 U_4 U_5 U_6 + \frac { (-9 - 5 \ii \sqrt{7})}{32} U_7 U_8 U_9 
  
   \\[.4em]&
  + 
  \frac{ (-7 - 3 \ii \sqrt{7}) }{32} U_0 ( U_1 U_7 + U_2 U_8 + U_3 U_9) + 
\frac{(1 + \ii \sqrt{7})}8 ( U_1^2 U_8 + U_2^2 U_9 + U_3^2 U_7) + (U_1 U_6^2 + 
    U_2 U_4^2 + U_3 U_5^2)   
      
      \\[.8em]
 eq_4=&
 \frac { (1 + \ii \sqrt{7})}8 U_0 U_3 (4 U_6 + U_9) - U_1 U_2 U_3 + 2 U_1^2 U_5 + 
 2 U_3 U_5^2 - 2 U_5 U_6  U_7 - U_5 U_7 U_9 - U_1 U_6 U_9
 
\\[.4em]

eq_5=& g_3(eq_4)
\\[.4em]
eq_6=&g_3^2(eq_4)

\\[.8em]
eq_7=&
(5 - \ii \sqrt{7}) U_1 U_2 U_3 + 16 U_4 U_5 U_6 + 
\frac {(-1 - \ii \sqrt{7})}2 U_0 U_1 U_7 - 
 2 U_0 U_2 (2 U_5 + U_8) + (1 - \ii \sqrt{7}) U_2 U_4 U_7 
 \\[.4em]&
 + (-1 +  \ii \sqrt{7} ) U_1^2 U_8 + 4 U_3 U_5 U_8 + (-2 + 2 \ii \sqrt{7} ) U_4 U_6 U_8 + 
(1 - \ii \sqrt{7}) U_6 U_7 U_8 + 2 U_3 U_8^2
 
 \\[.4em]

eq_8=& g_3(eq_7)
\\[.4em]
eq_9=&g_3^2(eq_7)

 \\[.8em]
eq_{10}=&
\frac{ (1 + 3 \ii \sqrt{7} ) }4 U_1 U_2 U_3 + \frac {(-5 + \ii \sqrt{7})}2 U_1^2 U_5 - 
 U_0 U_2 U_5 + (-1 + \ii \sqrt{7}) U_3 U_5^2 - 
 U_0 U_3 U_6 + (-5 + \ii \sqrt{7} ) U_4 U_5 U_6 
 
 \\[.4em]&
 + 
\frac { (3 - \ii \sqrt{7} )}4 U_0 U_1 U_7 - U_1^2 U_8 + 
\frac{ (-1 - \ii \sqrt{7})}8 U_0 U_2 U_8 + (1 + \ii \sqrt{7}) U_3 U_5 U_8 - 
 4 U_4 U_6 U_8  + \frac 12 (U_1 U_9^2)
 \\[.4em]&
 + \frac{(1 - \ii \sqrt{7}) }4 (2U_6 U_7 U_8+ 
 4U_5 U_7 U_9 + 4 U_5 U_6 U_7+  U_7 U_8 U_9  )+ 
\frac{ (3 + \ii \sqrt{7} )}4 U_3 U_8^2 - 2 U_4 U_5 U_9  + \frac{(-1 - \ii \sqrt{7})}2 U_4 U_8 U_9

\\[.4em]
eq_{11}=& g_3(eq_{10})
\\[.4em]eq_{12}=&g_3^2(eq_{10})
\\

eq_{13}=&-2 \ii \sqrt{7} U_1^2 U_3+\frac{ (-7+5 \ii \sqrt{7})}4 U_0 U_2 U_3+(-7+\ii \sqrt{7}) U_0 U_6^2+U_0^2 U_7+(2-2 \ii \sqrt{7}) U_1 U_4 U_7
\\[.4em]&
+(-5-\ii \sqrt{7}) U_2 U_5 U_7+(2+2 \ii \sqrt{7}) U_3 U_6 U_7+\frac{ (-1-5 \ii \sqrt{7})}4 U_1 U_7^2-2 U_2 U_7 U_8+\frac{ (3+3 \ii \sqrt{7})}2 U_3 U_7 U_9

\\[.8em]

eq_{14}=&U_1^2 U_3+\frac{ (3-\ii \sqrt{7})}4 U_0 U_1 U_5+2 U_3 U_4 U_6-2 U_5^2 U_6+\frac{ (1+\ii \sqrt{7})}4 U_2 U_5 U_7-U_3 U_6 U_7
\\[.4em]&
+\frac{ (-1-\ii \sqrt{7}) }4 U_3^2 U_8+\frac{ (-1+\ii \sqrt{7})}4 U_0 U_6 U_9+\frac{ (-5-\ii \sqrt{7})}8 U_3 U_7 U_9

\\[.8em]

eq_{15}=&\frac{ (-3-\ii \sqrt{7}) }2 U_1^2 U_3+\frac{ (3-\ii \sqrt{7}) }2 U_0 U_2 U_3+(-1+\ii \sqrt{7}) U_0 U_1 U_5+(-1-\ii \sqrt{7}) U_3^2 U_5
\\[.4em]&
+2 U_1 U_2 U_6+(1+\ii \sqrt{7}) U_0 U_6^2-U_0^2 U_7+\frac{(1+\ii \sqrt{7})}4 U_1 U_7^2+\frac{ (-1+\ii \sqrt{7})}2 U_0 U_1 U_8+U_3 U_7 U_9

\\[.8em]

eq_{16}=&(-3+\ii \sqrt{7}) U_2^3+(-3+\ii \sqrt{7}) U_1^2 U_3+4 U_0 U_2 U_3+(-2-2 \ii \sqrt{7}) U_0^2 U_4+8 U_1 U_4^2+8 U_0 U_1 U_5
\\[.4em]&
+(-5-\ii \sqrt{7}) U_1 U_2 U_6+(4+4 \ii \sqrt{7}) U_3 U_4 U_6+2 U_0 U_1 U_8+(3-\ii \sqrt{7}) U_2 U_7 U_8+(2+2 \ii \sqrt{7}) U_3 U_4 U_9

\\[.8em]

eq_{17}=&(-1-\ii \sqrt{7}) U_2^3+1/4 (5+\ii \sqrt{7}) U_0 U_2 U_3+(3-\ii \sqrt{7}) U_3^2 U_5+(4-4 \ii \sqrt{7}) U_2 U_4 U_5
\\[.4em]&
+(-1-\ii \sqrt{7}) U_2 U_5 U_7-2 U_1 U_2 U_9+(1+\ii \sqrt{7}) U_3 U_4 U_9-8 U_5^2 U_9-4 U_5 U_8 U_9

\\[.8em]

eq_{18}=&U_1^2 U_3+\frac{ (-5-\ii \sqrt{7})}8 U_0 U_2 U_3+\frac{ (1+\ii \sqrt{7})}2 U_3^2 U_5+\frac{ (1+\ii \sqrt{7}) }2 U_1 U_2 U_6
\\[.4em]&
+(-2+2 \ii \sqrt{7}) U_5^2 U_6+U_2 U_5 U_7-2 U_3 U_6 U_7+(-1+\ii \sqrt{7}) U_5 U_6 U_8-U_3 U_7 U_9

\\[.8em]

eq_{19}=&\frac{ (-5-\ii \sqrt{7})}2 U_0^2 U_4-4 U_2 U_5 U_7+\frac{(-1-\ii \sqrt{7})} U_1 U_7^2+2 U_0 U_1 U_8-2 U_2 U_7 U_8+\frac{ (-5+\ii \sqrt{7})}2 U_1 U_2 U_9
\\[.4em]&
+(1-\ii \sqrt{7}) U_3 U_4 U_9+(1-\ii \sqrt{7}) U_0 U_6 U_9+2 U_3 U_7 U_9+U_8^2 U_9+U_0 U_9^2

\\[.8em]

eq_{20}=&(1+\ii \sqrt{7}) U_1^2 U_3+1/2 (1-\ii \sqrt{7}) U_0 U_2 U_3-2 U_0^2 U_4+(-3-\ii \sqrt{7}) U_1 U_4^2-2 \ii \sqrt{7} U_0 U_1 U_5
\\[.4em]&
+(2-2 \ii \sqrt{7}) U_2 U_4 U_5+1/4 (5-\ii \sqrt{7}) U_0 U_1 U_8+1/2 (-5+\ii \sqrt{7}) U_3^2 U_8+4 U_5 U_8 U_9+2 U_8^2 U_9

\\[.8em]

eq_{21}=&(1-\ii \sqrt{7}) U_1^2 U_3-4 U_0 U_1 U_5-8 U_3 U_4 U_6-8 U_0 U_6^2+4 U_1 U_4 U_7+(2-2 \ii \sqrt{7}) U_2 U_5 U_7
\\[.4em]&
+2 U_1 U_7^2-2 U_0 U_1 U_8+(1+\ii \sqrt{7}) U_3^2 U_8+(1-\ii \sqrt{7}) U_2 U_7 U_8+(-1+\ii \sqrt{7}) U_3 U_7 U_9

\\[.8em]

eq_{22}=&-2 U_1^2 U_3+4 U_2 U_4 U_5-2 U_1 U_2 U_6+(1+\ii \sqrt{7}) U_0 U_6^2+\frac{ (1+\ii \sqrt{7})}4 U_0 U_1 U_8
\\[.4em]&

+2 U_2 U_4 U_8-2 U_5 U_6 U_8+U_1 U_2 U_9-2 U_3 U_4 U_9+\frac{ (1+\ii \sqrt{7}) }2U_0 U_6 U_9

\\[.8em]

eq_{23}=&\frac{ (-3+\ii \sqrt{7}) }4 U_2^3+\frac{ (-3+\ii \sqrt{7})}4 U_1^2 U_3+(-1-\ii \sqrt{7}) U_1 U_4^2+\frac{(-1+3 \ii \sqrt{7})}4 U_1 U_2 U_6
+\frac{ (-1-\ii \sqrt{7}) }2U_1 U_4 U_7
\\[.4em]&
+\frac{ (1+\ii \sqrt{7})}4 U_3^2 U_8+2 U_2 U_4 U_8+(-1+\ii \sqrt{7}) U_5 U_6 U_8+U_2 U_7 U_8+U_1 U_2 U_9

\\[.8em]

eq_{24}=&U_0 U_2 U_3+1/2 (-1-\ii \sqrt{7}) U_0^2 U_4+(1-\ii \sqrt{7}) U_0 U_1 U_5+(1-\ii \sqrt{7}) U_1 U_2 U_6+U_0 U_1 U_8
\\[.4em]&
-2 U_3^2 U_8-2 U_2 U_4 U_8+(1-\ii \sqrt{7}) U_5 U_6 U_8+2 U_5 U_8 U_9+U_8^2 U_9

\\
 \hline    
\end{array}
}
$$

\end{center}
\label{eqs1-24}
\end{table}

\begin{table}[htp]
\caption{Equations of the fake projective plane 25-84.}
\begin{center}
$$
\tiny{
\begin{array}{|cl|}
\hline
eq_{25}=&(-1+3 \ii \sqrt{7}) U_0^2 U_1+(44-4 \ii \sqrt{7}) U_2^2 U_3+64 U_3 U_4 U_5+(36-12 \ii \sqrt{7}) U_1 U_3 U_6+(16+16 \ii \sqrt{7}) U_4^2 U_6
\\[.4em]&
+(-4-4 \ii \sqrt{7}) U_0 U_2 U_7-32 U_3 U_4 U_8+(4+4 \ii \sqrt{7}) U_0 U_6 U_8-16 U_3 U_7 U_8+(8-8 \ii \sqrt{7}) U_1 U_3 U_9+16 U_4 U_7 U_9

\\[.8em]

eq_{26}=&(-1+3 \ii \sqrt{7}) U_0^2 U_1+(-4-4 \ii \sqrt{7}) U_2^2 U_3+(40-8 \ii \sqrt{7}) U_1 U_2 U_5+(4-12 \ii \sqrt{7}) U_1 U_3 U_6+96 U_4^2 U_6
\\[.4em]&
+(-24-8 \ii \sqrt{7}) U_2 U_6^2+16 U_1^2 U_7+(-2+2 \ii \sqrt{7}) U_0 U_2 U_7+64 U_4 U_6 U_7+(20-4 \ii \sqrt{7}) U_1 U_2 U_8-8 U_0 U_6 U_8
\\[.4em]&
+16 U_4 U_7 U_9

\\[.8em]

eq_{27}=&(5+\ii \sqrt{7}) U_0^2 U_1+(-4-4 \ii \sqrt{7}) U_2^2 U_3+(16-16 \ii \sqrt{7}) U_3 U_4 U_5+(-20-4 \ii \sqrt{7}) U_1 U_3 U_6+32 U_4^2 U_6
\\[.4em]&
+32 U_0 U_5 U_6+8 U_0 U_6 U_8-16 U_1 U_3 U_9+16 U_0 U_5 U_9+8 U_0 U_8 U_9

\\[.8em]

eq_{28}=&8 U_2^2 U_3+(-3+\ii \sqrt{7}) U_0 U_3^2+(-4-4 \ii \sqrt{7}) U_1 U_2 U_5+(4+4 \ii \sqrt{7}) U_3 U_4 U_5+32 U_5^3+(4+4 \ii \sqrt{7}) U_3 U_5 U_7
\\[.4em]&
+16 U_5^2 U_8+(3-\ii \sqrt{7}) U_1 U_3 U_9+8 U_2 U_6 U_9

\\[.8em]

eq_{29}=&(-3+\ii \sqrt{7}) U_2^2 U_3+(5+\ii \sqrt{7}) U_0 U_2 U_4+8 U_1 U_2 U_5-8 U_2 U_6^2+2 U_0 U_2 U_7+(-1-\ii \sqrt{7}) U_1 U_2 U_8+8 U_5^2 U_8
\\[.4em]&
+(3-\ii \sqrt{7}) U_1 U_3 U_9+(4+4 \ii \sqrt{7}) U_4^2 U_9-8 U_2 U_6 U_9+(2+2 \ii \sqrt{7}) U_4 U_7 U_9-2 U_0 U_8 U_9+(-3+\ii \sqrt{7}) U_2 U_9^2

\\[.8em]

eq_{30}=&8 U_2^2 U_3+(4-4 \ii \sqrt{7}) U_1^2 U_4+(-12-4 \ii \sqrt{7}) U_1 U_2 U_5+(-4-12 \ii \sqrt{7}) U_4^2 U_6+(12+4 \ii \sqrt{7}) U_2 U_6^2
\\[.4em]&
+(2-2 \ii \sqrt{7}) U_1^2 U_7-8 U_1 U_2 U_8-16 U_3 U_4 U_8+(1+3 \ii \sqrt{7}) U_0 U_6 U_8+(-3-\ii \sqrt{7}) U_3 U_7 U_8+4 U_1 U_3 U_9
\\[.4em]&
+(6+2 \ii \sqrt{7}) U_2 U_6 U_9

\\[.8em]

eq_{31}=&(-4+4 \ii \sqrt{7}) U_1^2 U_4-4 U_1 U_2 U_5+(-4+4 \ii \sqrt{7}) U_3 U_4 U_5+16 U_5^3+(-8+8 \ii \sqrt{7}) U_4^2 U_6+(2+2 \ii \sqrt{7}) U_0 U_5 U_6
\\[.4em]&
-4 U_1^2 U_7+(2+2 \ii \sqrt{7}) U_6 U_7^2+8 U_3 U_4 U_8-4 U_0 U_6 U_8-4 U_5 U_8^2+(1+\ii \sqrt{7}) U_7^2 U_9

\\[.8em]

eq_{32}=&(-5-\ii \sqrt{7}) U_0^2 U_1+(-6+2 \ii \sqrt{7}) U_0 U_3^2+(-24+8 \ii \sqrt{7}) U_3 U_4 U_5+(20+4 \ii \sqrt{7}) U_1 U_3 U_6-32 U_4^2 U_6
\\[.4em]&
-32 U_0 U_5 U_6+32 U_2 U_6^2+(2+2 \ii \sqrt{7}) U_0 U_2 U_7+(4+4 \ii \sqrt{7}) U_1 U_2 U_8-8 U_0 U_6 U_8+(10+2 \ii \sqrt{7}) U_1 U_3 U_9
\\[.4em]&
+16 U_2 U_6 U_9

\\[.8em]

eq_{33}=&(7-5 \ii \sqrt{7}) U_0^2 U_1+(-56-24 \ii \sqrt{7}) U_1^2 U_4+32 \ii \sqrt{7} U_1 U_2 U_5+(28+4 \ii \sqrt{7}) U_1 U_3 U_6+(28+28 \ii \sqrt{7}) U_0 U_5 U_6
\\[.4em]&
+(-84-4 \ii \sqrt{7}) U_1^2 U_7+(7+7 \ii \sqrt{7}) U_0 U_2 U_7-56 U_3 U_5 U_7+56 U_6 U_7^2+24 \ii \sqrt{7} U_1 U_2 U_8+56 U_0 U_6 U_8
\\[.4em]&
+(14-18 \ii \sqrt{7}) U_1 U_3 U_9+28 U_7^2 U_9

\\[.8em]

eq_{34}=&(-5-\ii \sqrt{7}) U_0^2 U_1+48 U_1 U_2 U_5+(-16-16 \ii \sqrt{7}) U_3 U_4 U_5+32 U_4^2 U_6+(2+10 \ii \sqrt{7}) U_1^2 U_7
\\[.4em]&
+(-48+16 \ii \sqrt{7}) U_4 U_6 U_7
+(28-4 \ii \sqrt{7}) U_1 U_2 U_8+(-12-12 \ii \sqrt{7}) U_3 U_4 U_8+(-16-8 \ii \sqrt{7}) U_0 U_6 U_8
\\[.4em]&
+(-22+2 \ii \sqrt{7}) U_1 U_3 U_9
+(-8-8 \ii \sqrt{7}) U_2 U_6 U_9+(-8+8 \ii \sqrt{7}) U_4 U_7 U_9

\\[.8em]

eq_{35}=&(10+2 \ii \sqrt{7}) U_2^2 U_3+(-11+\ii \sqrt{7}) U_0 U_2 U_4-16 U_1 U_2 U_5+(20+4 \ii \sqrt{7}) U_3 U_4 U_5-16 U_2 U_6^2
\\[.4em]&
+(-1-\ii \sqrt{7}) U_0 U_2 U_7
+(-2-2 \ii \sqrt{7}) U_1 U_2 U_8-16 U_5^2 U_8+(-4-4 \ii \sqrt{7}) U_4^2 U_9+(3-\ii \sqrt{7}) U_2 U_9^2

\\[.8em]

eq_{36}=&(2+2 \ii \sqrt{7}) U_0 U_3^2+(-6+2 \ii \sqrt{7}) U_0 U_2 U_4+(4-4 \ii \sqrt{7}) U_1 U_3 U_6+32 U_4^2 U_6+(-12-4 \ii \sqrt{7}) U_2 U_6^2+2 U_0 U_2 U_7
\\[.4em]&
+16 U_4 U_6 U_7+(7-\ii \sqrt{7}) U_1 U_2 U_8-8 U_5^2 U_8+4 U_1 U_3 U_9+(4-4 \ii \sqrt{7}) U_0 U_5 U_9+4 U_4 U_7 U_9-2 \ii \sqrt{7} U_0 U_8 U_9

\\[.8em]

eq_{k}=&g_3(eq_{k-24}),~~k=37,\ldots,60

\\[.8em]
eq_{k}=&g_3^2(eq_{k-48}),~~k=61,\ldots,84

    \\
\hline
\end{array}
}
$$
\end{center}
\label{eqs25-84}
\end{table}

\end{document}